\documentclass{amsart}
\usepackage{amsmath}
\usepackage{amsfonts}
\usepackage{amssymb}
\usepackage{graphicx}
\usepackage{float, verbatim}
\usepackage{enumerate}
\usepackage{color}
\usepackage{hyperref}

\newcommand{\Haus}{\dim_{\mathrm{H}}}
\newcommand{\Boxd}{\dim_{\mathrm{B}}}
\newtheorem*{thm*}{Theorem}
\newtheorem*{conj*}{Conjecture}

\newtheorem{thm}{Theorem}[section]

\newtheorem{lma}[thm]{Lemma}

\newtheorem{defn}[thm]{Definition}

\newtheorem{rem}[thm]{Remark}

\begin{document}

\title{Kakeya books and projections of Kakeya sets}

\author{Han Yu}
\address{Han Yu\\
School of Mathematics \& Statistics\\University of St Andrews\\ St Andrews\\ KY16 9SS\\ UK \\ }
\curraddr{}
\email{hy25@st-andrews.ac.uk}
\thanks{}



\date{}

\dedicatory{}

\begin{abstract}
Here we show some results related with Kakeya conjecture which says that for any integer $n\geq 2$, a set containing line segments in every dimension in $\mathbb{R}^n$ has full Hausdorff dimension as well as box dimension. We proved here that the Kakeya books, which are Kakeya sets with some restrictions on positions of line segments have full box dimension. We also prove here a relation between the projection property of Kakeya sets and the Kakeya conjecture. If for any Kakeya set $K\subset\mathbb{R}^n$, the Hausdorff dimension of orthogonal projections on $k\leq n$ subspaces is independent of directions then the Kakeya conjecture is true. Moreover, the converse is also true.
\end{abstract}

\maketitle

\section{Introduction}\label{Intro}

A Kakeya set in $\mathbb{R}^n$ is a set which contains a unit line segment in every direction. There are good reference sources for this topic, for example in \cite{TT}, \cite{IL}, where we can find many recent results.

To clarify the conventions we use, consider the following restricted sense of Kakeya set:
\begin{defn}\label{KAKEYA}
	Let $\Omega\subset S^{n-1}$ be a set of directions, then a  set $K(\Omega)\subset\mathbb{R}^n$ is a $\Omega$-Kakeya set iff it is a bounded set and it is an union of lines in different direction:
	\[
	K(\Omega)=\bigcup_{\theta\in\Omega} l_\theta,
	\]
	where $l_\theta$ is a line segment with unit length centred at a point $a_\theta\in\mathbb{R}^n$ pointing in direction $\theta$.
\end{defn}

This definition is a bit more restrictive than other references in the sense that usually we only require that:
\[
K(\Omega)\supset\bigcup_{\theta\in\Omega} l_\theta.
\]
In this sense the definition \ref{KAKEYA} can be seen as Kakeya set 'without redundancy', namely, we can write $K(\Omega)$ as a union of line segments such that for every direction in $\Omega$ there is an unique line segment with that direction. Then we refer Kakeya set as following:
\begin{defn}\label{KAKEYA1}
	A set $K\subset\mathbb{R}^n$ is a Kakeya set iff there is a set of directions $\Omega\subset S^{n-1}$ with non empty interior such that $K$ is a $\Omega$-Kakeya set.
\end{defn}

A long standing conjecture is that for integer $n\geq 2$ any Kakeya set $K\subset\mathbb{R}^n$ has Hausdorff dimension $n$. It is true for Kakeya set in $n=2$ \cite{DA} and for $n\geq 3$ several partial results can be found, for example see the pre-mentioned websites \cite{TT}, \cite{IL}. Or more specifically in \cite{Ka}. In this paper we consider a special type of Kakeya set. In order to define it precisely we need the following topological considerations:

In $\mathbb{R}^n$, the direction of lines can be represented by $S^{n-1}\subset\mathbb{R}^n$, where:
\[
S^{n-1}=\big\{x\in\mathbb{R}^n| \|x\|=1\big\},
\]
where $\|.\|$ is the Euclidean norm. Then we consider the $k$-hyperplanes $H^k$ spanned by $k\leq n$ linearly independent vectors. The orientations of $(n-1)$-hyperplanes intersecting $(0,\dots,0)$ can be also represented by $S^{n-1}$. The usual way of doing this is to identify the orientation of a $(n-1)$-hyperplane with the direction of lines that are orthogonal to that hyperplane. For any $2$-hyperplane $H^{2}$, we can consider the intersection:
\[
H^{2}\cap S^{n-1},
\]
the intersection can be identified as a circle $S^1\subset H^{2}$. Then for any $v\in H^{2}\cap S^{n-1}$, there is a $(n-1)$-hyperplane $H^{n-1}(v)$ orthogonal to $v$. We shall refer following set as \emph{pages} with respect to $H^2$:
\[
P(H^2)=\big\{H^{n-1}(v)| v\in H^2\cap S^{n-1} \big\}.
\] 
We refer $v$ as the directional angle of the hyperplane $H^{n-1}(v)$. Or more conveniently, we can isometrically parametrize  $\phi:H^{2}\cap S^{n-1}\to [0,2\pi)$ and refer the direction angle as the corresponding number $\phi(v)$ in $[0,2\pi)$. There are different isometric parametrizations $\phi$, but different such parametrizations differ each other by shifting a constant $\mod 2\pi$. Namely for any two isometric parametrizations $\phi,\phi'$ there is a constant $c\geq 0$ such that:
\[
\phi(v)=\phi'(v)+c\mod 2\pi.
\]
We shall fix those isometric parametrizations once for all and later we will state results without mentioning the parametrization.
\begin{defn}[Kakeya Book]
We define $n$-Kakeya book for integer $n$ inductively.

For $n=2$, a $2$-Kakeya book is a bounded Kakeya set in $\mathbb{R}^2$.

For $n=k+1$ where $k\geq 2$, if there exist a $2$-hyperplane $H^2$ such that for any $k$-hyperplane $H^{k}$ in the \emph{pages} $P(H^2)$ we can find a vector $a\in\mathbb{R}^{k+1}$ such that the affine hyperplane $H^{k}+a$ contains a bounded $k$-Kakeya book. Furthermore, all those $k$-Kakeya books are uniformly bounded in the sense that there is a universal bound $C>0$ such that all $k$-Kakeya books are contained in $B(0,C)$.
\end{defn}
\begin{rem}
For $n=3$, the normal direction of $2$ hyperplane can be viewed as $S^2$, and for a great circle $S^1\subset S^2$, and for any $\theta\in S^1$ we can find a $2$ dimensional affine hyperplane $H_\theta$ orthogonal with direction $\theta$ such that $H_\theta$ contains a $2$-Kakeya book, which is a Kakeya set in $\mathbb{R}^2$. For example the unit ball $B(0,1)$ is a $3$-Kakeya book.
\end{rem}

It turns out that Kakeya books always have full box dimension, namely:
\begin{thm}\label{book}[Kakeya book theorem]
For any $n\geq 2$, any $n$-Kakeya book has box dimension $n$.
\end{thm}

Of course Kakeya books are very special type of Kakeya sets in the sense we require more on the configuration of lines. For discussions on box dimensions including precise definitions and properties see \cite[chapter 3]{Fa}.

Another consideration here is the projection property of Kakeya sets. Intuitively speaking, any orthogonal projection of a Kakeya set is also Kakeya set. In this sense, a Kakeya set should carry some orientation invariance.  Given some recent results on projection of self similar sets \cite{SH}, \cite{Fal} it is plausible to make the following conjecture:
\begin{conj*}[Projection Conjecture]
	Let $K\subset\mathbb{R}^n$ be a Kakeya set, then for any integer $0<k<n$ and for \emph{all} $\gamma\in G(n,k)$ the Hausdorff dimension is constant, namely there is a number $c$ such that:
	\[
	\Haus (\pi_\gamma K)=c.
	\]
\end{conj*}

Here $G(n,k)$ is the Grassmannian manifold which represents the orientations of $k$-subspaces of $\mathbb{R}^n$. $\pi_\gamma$ for $\gamma\in G(n,k)$ is the orthogonal projection onto a $k$-subspace oriented in $\gamma$. A result of Marstrand and Mattila \cite[chapter 3]{Ma} says that for any Borel set $F\subset\mathbb{R}^n$, the Hausdorff dimension $\dim_{\mathrm{H}} \pi_\gamma (F)$ of almost all orthogonal projections $\gamma\in G(n,k)$ onto $k$-subspaces is equal to:
\[
\max\{k,\dim_{\mathrm{H}} F\}.
\]
And we have the following theorem:
\begin{thm}\label{Th1}
	Kakeya conjecture $\iff$ Projection conjecture.
\end{thm}
\section{Proof of Kakeya book theorem}
We know that for $n=2$ any $2$-Kakeya book has Hausdorff dimension $2$, however we need a more precise quantitative result:
\begin{lma}\label{L1}
In $\mathbb{R}^2$. Let $0<c\leq 2\pi$ be a positive number and $\epsilon>0$ be any positive number, further let $\alpha<1$ be any positive number. Now suppose that we have $(1\times\epsilon)$-rectangles whose long side have directional angles $0,\epsilon^\alpha,2\epsilon^\alpha,3\epsilon^\alpha,\dots,[c/\epsilon^\alpha]\epsilon^\alpha$ (the positions can be arbitrary, and we shall refer such rectangles as tubes). Then there is a positive number $A>0$ such that for all $\epsilon>0$, the union of those tubes have Lebesgue measure at least:
\[
|\bigcup_i T_i|\geq A\epsilon^{1-\alpha}.
\]
Here $T_i\subset\mathbb{R}^2$ are the tube with direction $i\epsilon^{\alpha}, i\in\{0,1,\dots,[c/\epsilon^\alpha]\}$ and $[.]$ denotes the integer part of a positive real number.
\end{lma}
\begin{rem}
This implies that any $2$-Kakeya book/Kakeya set has box dimension $2$.
\end{rem}
\begin{proof}
Let $T_i$ be the tube with directional angle $i\epsilon^{\alpha}$. Then we see that for any two tubes $T_i, T_j$ the Lebesgue measure of their intersection can be bounded from above:
\[
|T_i\cap T_j|\leq \frac{1}{\sin (|i-j|\epsilon^\alpha)}\epsilon^2\leq \frac{2}{\pi}\frac{1}{|i-j|\epsilon^\alpha}\epsilon^2.
\]
So we see that there is are constants $C,C'>0$:
\begin{eqnarray*}
\sum_{i,j} |T_i\cap T_j|&\leq&\sum_{i,j} \frac{2}{\pi}\frac{1}{|i-j|}\epsilon^{2-\alpha}\leq\frac{2}{\pi}\epsilon^{2-\alpha}\bigg[\frac{c}{\epsilon^\alpha}\bigg]\sum_{i}\frac{1}{i}\leq C'\frac{2}{\pi}\epsilon^{2-\alpha}\bigg[\frac{c}{\epsilon^\alpha}\bigg]\log \bigg[\frac{c}{\epsilon^\alpha}\bigg]\\
&\leq& C\epsilon^{2-2\alpha}\log\frac{1}{\epsilon}.
\end{eqnarray*}
Using Bonferroni inequalities we see that there are constants $C'', A>0$ such that for all $\epsilon>0$ :
\begin{eqnarray*}
|\bigcup_i T_i|&\geq&\sum_i |T_i|-\sum_{i,j} |T_i\cap T_j|\\
&\geq& \bigg[\frac{c}{\epsilon^\alpha}\bigg]\epsilon-C\epsilon^{2-2\alpha}\log\frac{1}{\epsilon}\\
&\geq& C''\epsilon^{1-\alpha}-C\epsilon^{2-2\alpha}\log\frac{1}{\epsilon}\\
&=&  \epsilon^{1-\alpha}\left(C''-C\epsilon^{1-\alpha}\log \frac{1}{\epsilon}\right)\\
&\geq& A\epsilon^{1-\alpha}.
\end{eqnarray*}
\end{proof}
The above lemma established the Kakeya book theorem in $\mathbb{R}^2$. To proceed further, we will now consider some technical lemmas for $(n-1)$-hyperplanes in $\mathbb{R}^n$. In what follows, for a set $F\subset\mathbb{R}^n$ we use $A^{\epsilon}$ with an $\epsilon>0$ to denote the $\epsilon$-neighbourhood of $A$:
\[
A^{\epsilon}=\big\{x\in\mathbb{R}^n| \inf_{a\in A}\|x-a\|<\epsilon\big\}.
\]

\begin{lma}\label{lm1}
Let $T_1, T_2$ be two $n-1$-hyperplanes in $\mathbb{R}^n$ such that the normal directions $\theta_1, \theta_2$ of $T_1, T_2$ represented as elements in $S^{n-1}$ has spherical distance $d>0$. Then for $\epsilon>0$, the intersection $T^{\epsilon}_1\cap T^{\epsilon}_2$ is isometric to( image of rotation and translation) $ E\times\mathbb{R}^{n-1}$, where $E\subset\mathbb{R}^2$ is a parallelogram with $2$ dimensional Lebesgue measure $\frac{1}{\sin d}\epsilon^2$.
\end{lma}
\begin{proof}
Without loss of generality we can consider the plane $H$ spanned by two vectors $\theta_1, \theta_2$, then set up coordinates as $\mathbb{R}^n=H\times\mathbb{R}^{n-2}$.   Then $T^{\epsilon}_1=S_1\times \mathbb{R}^{n-2}$, where $S_1$ is a strip with width $\epsilon>0$. Similarly $T^{\epsilon}_2=S_2\times \mathbb{R}^{n-2}$ for another strip $S_2$. $S_1, S_2$ are two strips with width $\epsilon>0$ and angle $d$ therefore $E=S_1\cap S_2$ and $|E|=\frac{1}{\sin d}\epsilon^2$ and the lemma concludes.
\end{proof}

\begin{lma}
Let $n>0$ be any positive integer, let $0<c\leq 2\pi$ be a positive number, $\epsilon>0$ be any positive number and let $\alpha<1$ be any positive number. Now suppose that there are \emph{pages} $P(H^2)$ such that for $(n-1)$-hyperplane $H_i^{n-1}\in P(H^2)$ with directional angle $i\epsilon^\alpha,i\in\{0,1,\dots,[c/\epsilon^\alpha]\}$ there is a $v_i\in\mathbb{R}^n$ and a $(1^{n-1}\times\epsilon)$-rectangle $T_i$ such that $T_i\cap H_i^{n-1}$ is a $1^{n-1}$-cube ( An embedded cube in $H_i^{n-1}+v_i$ with side length $1$) . Then there is a positive number $A>0$ such that for all $\epsilon>0$, the union of those rectangles have Lebesgue measure ($n$ dimensional) at least:
\[
|\bigcup_i T_i|\geq A\epsilon^{1-\alpha}.
\]
\end{lma}
\begin{proof}
The proof is almost the same as lemma \ref{L1}. The difference is that we use lemma \ref{lm1} to bound the measure of intersection:
\[
T_i\cap T_j
\]
by noticing that the spherical distance of directions $T_i, T_j$ is precisely:
\[
|i-j|\epsilon^\alpha.
\]
\end{proof}

We can now prove the Kakeya book theorem:

\begin{proof}[Proof of theorem \ref{book}]
We have established this theorem for $n=2$ in lemma \ref{L1}. To see the result for $n=3$, let $0<\alpha, \beta<1$ be two arbitrary numbers. Then we can, by definition of $3$-Kakeya book choose $[2\pi/\epsilon^\beta]+1$ many $2$-hyperplanes $H_i,i=0,1,2...,[2\pi/\epsilon^\beta]$ with directional angles $i \epsilon^\beta,i=0,1,...,[2\pi/\epsilon^\beta]$ such that all those hyperplanes contains a $2$-Kakeya book. Then in each $H_i$ we can choose $[2\pi/\epsilon^\beta]+1$ lines $T_{ij},j=0,1,...,[2\pi/\epsilon^\beta]$ whose directional angles are equally distributed (similar as in the condition in lemma \ref{L1}).

By lemma \ref{L1} we see that for any $i$:
\[
|\bigcup_j T^{\epsilon}_{ij}|\geq A\epsilon^{1-\alpha}\epsilon.
\]
Then we see that for any $i_1, i_2$:
\[
|\bigcup_j T^{\epsilon}_{i_1j} \cap \bigcup_j T^{\epsilon}_{i_2j}|\leq C\frac{2}{\pi}\frac{1}{|i-j|}\epsilon^{2-\beta}.
\]
Then same argument as in the proof of lemma \ref{L1} we see that (with $C_1,C_2>0$ positive constants):
\[
|\bigcup_{i,j}T^{\epsilon}_{ij}|\geq C_1\epsilon^{2-\alpha-\beta}-C_2\epsilon^{2-2\beta}\log \frac{1}{\epsilon},
\]
notice that we can choose $\beta=\alpha^2$ and by letting $\alpha\to 1$ we see that since $|\bigcup_{i,j}T^{\epsilon}_{ij}|\subset K^\epsilon$:

\[
\Boxd K=3.
\]

The result for general $n\geq 3$ follows by induction with the help of lemma \ref{lm1}.
\end{proof} 
\section{Proof of equivalence of Kakeya conjecture and projection conjecture}
 In this section we prove the equivalence of projection conjecture and the Kakeya conjecture mentioned in section \ref{Intro}. We first need a simple lemma:
\begin{lma}\label{spa}[Spaghetti Lemma]
	Let $\Omega\subset S^{n-1}$ be a open set such that and $K$ be a $\Omega$-Kakeya set. Suppose further that $\forall\theta\in\Omega$ the unit segment  $l_\theta$ are contained in a line that goes through the origin, then $K$ contains non empty interior.
\end{lma}
\begin{proof}
	We can decompose $\mathbb{R}^{n}$ as annuli with thickness $\frac{1}{2}$:
	\[
	\mathbb{R}^{n}=\bigcup_{i=0}^{\infty} B(0,\frac{i+1}{2})\setminus B(0,\frac{i}{2}),
	\]
	as $K$ is bounded, therefore it intersects only finitely many annuli. Namely there is a positive constant $N>0$ such that:
	\[
	\bigcup_{\theta\in\Omega}\{a_\theta\}\subset \bigcup_{i=0}^{i=N} B(0,\frac{i+1}{2})\setminus B(0,\frac{i}{2}).
	\]
	Now define the following sets:
	\[
	\Omega_k=\{\theta\in S^{n-1}| l_\theta\cap \left(B(0,\frac{k+1}{2})\setminus B(0,\frac{k}{2})\right)\neq\emptyset \},
	\]
	we have that:
	\[
	\bigcup_{k=0}^{k=K} \Omega_k=\Omega,
	\]
	Baire category theorem implies that at least one of the $\Omega_k$ contains non empty interior $\Omega_0\subset\Omega$, then it is easy to see that the open sector:
	\[
	\bigg\{x\in B(0,\frac{k+1}{2})\setminus B(0,\frac{k}{2})| \text{ the line between $0$ and $x$ has direction in }\Omega_0\bigg\}\subset K
	\]
\end{proof}

Then by using a lifting and projecting method we can show the following result:

\begin{thm}\label{MAIN}
	Let $K(S^{n-1})\subset\mathbb{R}^n$ be an $S^{n-1}$-Kakeya set:
	\[
	K(S^{n-1})=\bigcup_{\theta\in S^{n-1}} l_\theta.
	\]
	Then there is a set $\tilde{K}\subset\mathbb{R}^{2n-1}$ such that for some projection $\gamma_0\in G(2n-1,n)$:
	\[
	\pi_{\gamma_0}\tilde{K}\subset K(S^{n-1}).
	\]
	Further for any $\gamma\in G(2n-1,n)$ the projection:
	\[
	\pi_\gamma\tilde{K}\subset\mathbb{R}^n
	\]
	is a Kakeya set and for almost all $\gamma\in G(2n-1,n)$:
	\[
	\pi_\gamma\tilde{K}\subset\mathbb{R}^n
	\]
	is a Kakeya set of Hausdorff dimension $n$. This implies in particular for any $\epsilon>0$, there is a rearrangement of the line segments, namely for any $l_\theta$ there is another line segment $l_{\theta'}$ such that:
	\[
	\max_{x\in l_{\theta'},y\in l_\theta}\{\|x-y\|\}<\epsilon,
	\]
	and furthermore:
	\[
	\Haus \bigcup_{\theta\in S^{n-1}} l'_\theta(a_\theta)=n.
	\]
\end{thm}
\begin{rem}
	This theorem says that for any Kakeya set, we can slightly rearrange the line segments consisting this Kakeya set to obtain another Kakeya set with full dimension. This result can be compared with \cite[Proposition 5.2]{Fr}. Our result is stronger in the sense that the rearrangement can be chosen to be projections. In fact as proved in \cite{Fr}, for any $\epsilon>0$, if we allow the rearrangement to be arbitrarily in the sense that for any line segment $l_\theta$ we can choose any line segment within distance smaller than $\epsilon$ then we can always obtain a Kekaya set with positive measure.
\end{rem}
\begin{rem}
	The requirement of $K(S^{n-1})$ being a $S^{n-1}$-Kakeya set is redundant but helps us remove some technical discussion, in fact it is enough to consider Kakeya sets in the sense of definition \ref{KAKEYA1}.
\end{rem}

\begin{proof}
	Here for convenience, for each line segment $l_\theta$, we use $a_\theta$ to denote the middle point and write $l_\theta(a_\theta)$ as $l_\theta$ to emphasize the middle point. Let $K=\bigcup_{\theta\in S^{n-1}} l_\theta(a_\theta)$ be the Kakeya set in consideration. Extending all line segments to infinite lines we see that any $(n-1)$-hyperplane $H$ intersects all the lines with direction not parallel with $H$. From now on we fix such a hyperplane $H$. Dropping those lines that are parallel with $H$, we get a smaller Kakeya set with line segments of direction not parallel with $H$. In fact we can drop more line segments by only keeping the line segments with directions in a open ball $\Omega\subset S^{n-1}$ that is disjoint from the parallel directions with $H$. We can give a smooth parametrization on $\Omega$ such that we can view it as the unit open ball $B^{n-1}(0,1)$ in $\mathbb{R}^{n-1}$:
	\[
	\Omega\approx B^{n-1}(0,1),
	\]
where '$\approx$' means smooth diffeomorphism. That is if $A\approx B$ then there is a smooth diffeomorphism $\Phi$ such that:
\[
\Phi(A)=B.
\]

	We denote this smaller set as $K(\Omega)$ which is a $\Omega$-Kakeya set.
	
	Now we consider the $2n-1$ dimensional Euclidean space $\mathbb{R}^{2n-1}$ and suppose that our original Kakeya set lies in the $n$-subspace with all last $n-1$ coordinates $0$:
	\[
	K(\Omega)\subset\mathbb{R}^n\times 0^{n-1}\subset \mathbb{R}^{2n-1}.
	\]
	Further we can set up the coordinate system of $\mathbb{R}^n$ as $\mathbb{R}\times H$. Then we lift any line segment $l_\theta(a_\theta)$  as follows:
	\[
	\tilde{l}_{\theta}(a_\theta)=l_{\theta}(a_\theta)+(0,0,0...0,x_\theta),
	\]
	where $(0,0,0...0,x_\theta)$ denotes the vector with first $n$ coordinates $0$ and last $n-1$ coordinates the same as $x_\theta$ which is the intersection between the extended infinite line and the hyperplane $H$.
	After this lifting we obtain a lifted set:
	\[
	\tilde{K}=\bigcup_{\theta\in\Omega}\tilde{l}_{\theta}(a_\theta)\subset\mathbb{R}^{2n-1}.
	\]
	
	Recall we use the following coordinate system:
	\[
	\mathbb{R}^{2n-1}=\mathbb{R}\times\mathbb{R}^{n-1}\times\mathbb{R}^{n-1}=\mathbb{R}\times H\times H.
	\]
	Then the $(n-1)$-hyperplane $\tilde{H}$ intersect all the extended lifted lines where
	\[
	\tilde{H}=\{(0,x_h,x_h')\in \mathbb{R}\times H\times H| x_h=x_h'\}.
	\]
	Consider the orthogonal complement $\tilde{H}^{\perp}$ of $\tilde{H}$ in $\mathbb{R}^{2n-1}$:
	\[
	\mathbb{R}^{2n-1}=\tilde{H}\times\tilde{H}^{\perp},
	\]
	then we have the orthogonal projection $\pi_{\tilde{H}^{\perp}}$.
	
	The image:
	\[
	\pi_{\tilde{H}^{\perp}}(\tilde{K})
	\]
	consists of line segments whose extended infinite lines intersect the origin. It is clear that the projection $\pi_{\tilde{H}^{\perp}}(\tilde{l}_{\theta}(a_\theta))$ is a line segment with direction $F(\theta)$ centred at $b_\theta$ for some smooth function $F:\Omega\to S^{n-1}$ and therefore the set of directions of image segments contains open set and all $b_\theta$ are bounded. Applying lemma \ref{spa} we see that  $\pi_{\tilde{H}^{\perp}}(\bigcup_{\theta\in F(\Omega)}\tilde{l}_{\theta}(a_\theta))$ has non empty interior and therefore has Hausdorff dimension $n$, this implies that:
	\[
	\Haus\bigcup_{\theta\in F(\Omega)}\tilde{l}_{\theta}(a_\theta)\geq n,
	\]
	then for almost every direction $\gamma\in G(2n-1,n)$ the projection $\pi_\gamma (\bigcup_{\theta\in F(\Omega)}\tilde{l}_{\theta}(a_\theta))$ has full Hausdorff dimension (see \cite[chapter 3]{Ma}). If we choose $\gamma$ such that $\gamma (x,x_h,x_h')=(x,x_h)$ then we see that $\pi_\gamma (\bigcup_{\theta\in F(\Omega)}\tilde{l}_{\theta}(a_\theta))=K(\Omega)$. Of course we do not know whether this particular $\gamma$ falls into the exceptional direction for projection, but we can choose $\gamma'$ arbitrarily near to $\gamma$ such that:
	\[
	\Haus \pi_{\gamma'} (\bigcup_{\theta\in F(\Omega)}\tilde{l}_{\theta}(a_\theta))=n.
	\]
	Observe that $ \bigcup_{\theta\in F(\Omega)}\tilde{l}_{\theta}(a_\theta)$ is actually a union of line segments which can be seen as a rearrangement of the original Kakeya set $K(\Omega)$ and since we can choose $\gamma'$ close to $\gamma$ all the distance conditions can be satisfied.
	
\end{proof}

We will now prove theorem \ref{Th1}:

\begin{proof}[Proof of theorem \ref{Th1}]
The statement $1\implies 2$ is trivial by observing that all projection of Kakeya sets are also Kakeya sets.

Now we assume statement $2$ and let $K\subset\mathbb{R}^n$ be any Kakeya set, then by theorem \ref{MAIN} we see that the lifted set $\tilde{K}$ is contained in a Kakeya set. In fact it is an union of line segments with certain directions and we can easily extend this set by adding line segments of other directions without changing the image of projection in the direction which returns $K$. Then by the constant projection property, any projection including the original set has full Hausdorff dimension.
\end{proof}

\bibliography{WKakeya}
\bibliographystyle{amsalpha}

\end{document}